\documentclass{amsart}
\usepackage{a4wide}
\usepackage[all]{xy}
\usepackage{color}

\newcommand{\w}{\omega}
\newcommand{\IN}{\mathbb N}
\newcommand{\Ra}{\Rightarrow}
\newcommand{\F}{\mathcal F}
\newcommand{\U}{\mathcal U}

\newcommand{\IR}{\mathbb R}

\newcommand{\A}{\mathcal A}

\newtheorem{problem}{Problem}
\newtheorem{theorem}{Theorem}
\newtheorem{proposition}{Proposition}
\newtheorem{corollary}{Corollary}

\newtheorem{claim}{Claim}

\theoremstyle{definition}
\newtheorem{example}{Example}

\theoremstyle{remark}
\newtheorem{remark}{Remark}

\title{Embeddings into countably compact Hausdorff spaces}
\author{Taras Banakh}
\address{T.Banakh: Ivan Franko National University of Lviv (Ukraine) and Jan Kochanowski University in Kielce (Poland)}
\email{t.o.banakh@gmail.com}
\author{Serhii Bardyla}
\address{S. Bardyla: Institute of Mathematics, Kurt G\"{o}del Research Center, Vienna, Austria}
\thanks{The work of the second author is supported by the Austrian Science Fund FWF (Grant  I
3709 N35).}
\email{sbardyla@yahoo.com}
\author[A.~Ravsky]{Alex Ravsky}
\address{A. Ravsky: Pidstryhach Institute for Applied Problems of Mechanics and Mathematics, Lviv, Ukraine}
\email{alexander.ravsky@uni-wuerzburg.de}
\subjclass{54D10; 54D30; 54D35; 54D80}
\begin{document}
\begin{abstract}
In this paper we consider the problem of characterization of topological spaces that embed into countably compact Hausdorff  spaces. We study the separation axioms of subspaces of  countably compact Hausdorff spaces and construct an example of a regular separable scattered topological space which cannot be embedded into an Urysohn countably compact topological space but embeds into a Hausdorff countably compact space.
\end{abstract}
\maketitle

It is well-known that a topological space $X$ is homeomorphic to a subspace of a compact Hausdorff space if and only if the space $X$ is Tychonoff.

In this paper we discuss the following  problem.

\begin{problem}
Which topological spaces are homeomorphic to subspaces of countably compact Hausdorff spaces?
\end{problem}

A topological space $X$ is
\begin{itemize}
\item {\em compact} if each open cover of $X$ has a finite subcover;
\item {\em $\w$-bounded} if each countable set in $X$ has compact closure in $X$;
\item {\em countably compact} if each sequence in $X$ has an accumulation point in $X$;
\item {\em totally countably compact} if each infinite set in $X$ contains an infinite subset with  compact closure in $X$.
\end{itemize}
These properties relate as follows:
$$\xymatrix{
\mbox{compact}\ar@{=>}[r]&\mbox{$\w$-bounded}\ar@{=>}[r]&\mbox{totally countably compact}\ar@{=>}[r]&\mbox{countably compact}.
}$$

More information on various generalizations of the compactness can be found in \cite{JvM,JSS,JV,GR,N,V}.

In this paper we establish some properties of subspaces of countably compact Hausdorff spaces and hence find some necessary conditions of embeddability of topological spaces into Hausdorff countably compact spaces. Also, we construct an example of regular separable first-countable scattered topological space which cannot be embedded into a Urysohn countably compact topological space but embeds into a totally countably compact Hausdorff space. 

First we recall some results \cite{BBR} on embeddings into $\w$-bounded spaces.


We recall \cite[\S3.6]{Eng} that the Wallman compactification $W(X)$ of a topological space $X$ is the space of closed ultrafilters, i.e., families $\U$ of closed subsets of $X$ satisfying the following conditions:
\begin{itemize}
\item $\emptyset\notin\U$;
\item $A\cap B\in\U$ for any $A,B\in\U$;
\item a closed set $F\subset X$ belongs to $\mathcal U$ if $F\cap U\ne\emptyset$ for every $U\in\U$.
\end{itemize}
The Wallman compactification $W(X)$ of $X$ is endowed with the topology generated by the base consisting of the sets
$$\langle U\rangle=\{\F\in W(X):\exists F\in\F,\;F\subset U\}$$ where $U$ runs over open subsets of $X$.
By (the proof of) Theorem~\cite[3.6.21]{Eng}, for any topological space $X$ its Wallman compactification $W(X)$ is compact.

Let $j_X:X\to W(X)$ be the map assigning to each point $x\in X$ the principal ultrafilter consisting of all closed sets $F\subset X$ containing the point $x$. It is easy to see that the image $j_X(X)$ is dense in $W(X)$. By \cite[3.6.21]{Eng}, for a $T_1$-space $X$ the map $j_X:X\to W(X)$ is a topological embedding.

In the Wallman compactification $W(X)$, consider the subspace $$W_\w X={\textstyle\bigcup}\{\overline{j_X(C)}:C\subset X,\;|C|\le\w\},$$ which is the union of closures of countable subsets of $j_X(X)$ in $W(X)$. The space $W_\w X$ will be called the {\em Wallman $\w$-compactification} of $X$. 

Following \cite{BBR}, we define a topological space $X$ to be {\em $\overline\w$-normal} if for any closed separable subspace $C\subset X$ and any disjoint closed sets $A,B\subset C$ there are disjoint open sets $U,V\subset X$ such that $A\subset U$ and $B\subset V$.

The properties of the Wallman $\w$-compactification are described in the following theorem whose proof can be found in \cite{BBR}. 

\begin{theorem}\label{t:w} For any \textup{(}$\overline\w$-normal\/\textup{)} topological space $X$, its Wallman $\w$-compactification $W_\omega X$  is $\w$-bounded \textup{(}and Hausdorff\textup{)}.
\end{theorem}


A topological space $X$ is called
\begin{itemize}
\item {\em first-countable} at a point $x\in X$ if it has a countable neighborhood base at $x$;
\item {\em Fr\'{e}chet-Urysohn at a point $x\in X$} if for each subset $A$ of $X$ with $x\in\bar A$ there exists a sequence $\{a_n\}_{n\in\omega}\subset A$ that converges to $x$;
\item {\em regular at a point $x\in X$} if any neighborhood of $x$ contains a closed neighborhood of $x$.
\item {\em completely regular at a point $x\in X$} if for any neighborhood $U\subset X$ of $x$ there exists a continuous function $f:X\to[0,1]$ such that $f(x)=1$ and $f(X\setminus U)\subset\{0\}$.
\end{itemize}

If for each point $x$ of a topological space $X$ there exists a countable family $\mathcal{O}$ of open neighborhoods of $x$ such that $\bigcap \mathcal{O}=\{x\}$, then we shall say that the space $X$ has {\em countable pseudocharacter}.

\begin{theorem}\label{t:N1} Let $X$ be a subspace of a countably compact Hausdorff space $Y$. If $X$ is first-countable at a point $x\in X$, then $x$ is regular at the point $x$.
\end{theorem}

\begin{proof} Fix a countable neighborhood base $\{U_n\}_{n\in\IN}$ at $x$ and assume that $X$ is not regular at $x$. Consequently, there exists an open neighborhood $U_0$ of $x$ such that $\overline{V}\not\subset U_0$ for any neighborhood $V$ of $x$. Replacing each basic neighborhood $U_n$ by $\bigcap_{k\le n}U_k$, we can assume that $U_n\subset U_{n-1}$ for every $n\in\IN$. The choice of the neighborhood $U_0$ ensures that for every $n\in\IN$ the set $\overline{U}_n\setminus U_0$ contains some point $x_n$. Since the space $Y$ is countably compact and Hausdorff, the sequence $(x_n)_{n\in\w}$ has an accumulation point $y\in Y\setminus U_0$. By the Hausdorff property of $Y$, there exists a neighborhood $V\subset Y$ of $x$ such that $y\notin \overline{V}$. Find $n\in\w$ such that $U_n\subset V$ and observe that $O_y:=Y\setminus\overline{V}$ is a neighborhood of $y$ such that  $O_y\cap\{x_i:i\in\w\}\subset \{x_i\}_{i<n}$, which means that $y$ is not an accumulating point of the sequence $(x_i)_{i\in\w}$.
\end{proof}

\begin{remark} Example~6.1 from~\cite{BBR} shows that in Theorem~\ref{t:N1} the regularity of $X$ at the point $x$  cannot be improved to the complete regularity at $x$.
\end{remark}

\begin{corollary}\label{c1}
Let $X$ be a subspace of a countably compact Hausdorff space $Y$. If $X$ is first-countable, then $X$ is regular.
\end{corollary}

The following example shows that Theorem~\ref{t:N1} cannot be generalized over Fr\'{e}chet-Urysohn spaces with countable pseudocharacter.

\begin{example}\label{e3} There exists a Hausdorff space $X$ such that
\begin{enumerate}
\item $X$ is locally countable and hence has countable pseudocharacter;
\item $X$ is separable and Fr\'echet-Urysohn;
\item $X$ is not regular;
\item $X$ is a subspace of a totally countably compact Hausdorff space.
\end{enumerate}
\end{example}

\begin{proof} Choose any point $\infty\notin \w\times\w$ and consider the space $Y=\{\infty\}\cup(\w\times\w)$ endowed with the topology consisting of the sets $U\subset Y$ such that if $\infty \in U$, then for every $n\in\w$ the complement $(\{n\}\times \w)\setminus U$ is finite. The definition of this topology ensures that $Y$ is Fr\'echet-Urysohn at the unique non-isolated point $\infty$ of $Y$.

Let $\F$ be the family of closed infinite subsets of $Y$ that do not contain the point $\infty$. The definition of the topology on $Y$ implies that for every $F\in\F$ and $n\in\w$ the intersection $(\{n\}\times\w)\cap F$ is finite. By the Kuratowski-Zorn Lemma, the family $\F$ contains a maximal almost disjoint subfamily $\A\subset\F$. The maximality of $\A$ guarantees that each set $F\in\F$ has infinite intersection with some set $A\in\A$.

Consider the space $X=Y\cup\A$ endowed with the topology consisting of the sets $U\subset X$ such that $U\cap Y$ is open in $Y$ and for any $A\in\A\cap U$ the set $A\setminus U\subset\w\times\w$ is finite.

We claim that the space $X$ has the properties (1)--(4). The definition of the topology of $X$ implies that $X$ is separable, Hausdorff and locally countable, which implies that $X$ has countable pseudocharacter. Moreover, $X$ is first-countable at all points except for $\infty$. At the point $\infty$ the space $X$ is Fr\'echet-Urysohn (because its open subspace $Y$ is Fr\'echet-Urysohn at $\infty$).

The maximality of the maximal almost disjoint family $\A$ guarantees that each neighborhood $U\subset Y\subset X$ of $\infty$ has an infinite intersection with some set $A\in\A$, which implies that $A\in\overline{U}$ and hence $\overline{U}\not\subset Y$. This means that $X$ is not regular (at $\infty$).

In the Wallman compactification $W(X)$ of the space $X$ consider the subspace $Z:=X\cup W_\omega\A=Y\cup W_\omega\A$.
We claim that the space $Z$ is Hausdorff and totally countably compact. To prove that $Z$ is Hausdorff, take two distinct ultrafilters $a,b\in Z$. If the ultrafilters $a,b$ are principal, then by the Hausdorff property of $X$, they have disjoint neighborhoods in $W(X)$ and hence in $Z$. Now assume that one of the ultrafilters $a$ or $b$ is principal and the other is not. We lose no generality assuming that $a$ is principal and $b$ is not. If $a\ne\infty$, then we can use the regularity of the space $X$ at $a$ and prove that $a$ and $b$ have disjoint neighborhoods in $W(X)\supset Z$. So, assume that $a=\infty$. It follows from $b\in Z=X\cup W_{\w} \A$ that the ultrafilter $b$ contains some countable set $\{A_n\}_{n\in\w}\subset\A$.  Consider the set $$V=\bigcup_{n\in\w}\big(\{A_n\}\cup A_n\setminus\bigcup_{k\le n}\{k\}\times\w\big)$$and observe that $V$ has finite intersection with every set $\{k\}\times\w$, which implies that $Y\setminus V$ is a neighborhood of $\infty$. Then $\langle Y\setminus V\rangle$ and $\langle V\rangle$ are disjoint open neighborhoods of $a=\infty$ and $b$ in $W(X)$.

Finally, assume that both ultrafilters $a,b$ are not principal. Since $a,b\in W_{\w} \A$ are distinct, there are disjoint countable sets $\{A_n\}_{n\in\w},\{B_n\}_{n\in\w}\subset\A$ such that $\{A_n\}_{n\in\w}\in a$ and $\{B_n\}_{n\in\w}\in b$.
Observe that the sets $$V=\bigcup_{n\in\w}(\{A_n\}\cup A_n\setminus\bigcup_{k\le n}B_k)\mbox{ \ and \ }W=\bigcup_{n\in\w}(\{B_n\}\cup B_n\setminus\bigcup_{k\le n}A_k)$$ are disjoint and open in $X$. Then $\langle V\rangle$ and $\langle W\rangle$ are disjoint open neighborhoods of the ultrafilters $a,b$ in $W(X)$, respectively.

To see that $Z$ is totally countably compact, take any infinite set $I\subset Z$. We should find an infinite set $J\subset I$ with compact closure $\bar J$ in $Z$. We lose no generality assume that $I$ is countable and $\infty\notin I$. If $J=I\cap W_\omega\A$ is infinite, then  $\bar J$ is compact by the $\w$-boundedness of $W_\omega\A$, see Theorem~\ref{t:w}. If $I\cap W_\omega\A$ is finite, then $I\cap Z\setminus W_\omega\A=I\cap Y=I\cap(\w\times\w)$ is infinite. If for some $n\in\w$ the set $J_n=I\cap(\{n\}\times\w)$ is infinite, then $\bar J_n=J_n\cup\{\infty\}$ is compact by the definition of the topology of the space $Y$. If for every $n\in\w$ the set $I\cap(\{n\}\times\w)$ is finite, then $I\cap(\w\times\w)\in\F$ and by the maximality of the family $\A$, for some set $A\in\A$ the intersection $J=A\cap I$ is infinite, and then $\bar J=J\cup\{A\}$ is compact.
\end{proof}

A topological space $X$ is called {\em weakly $\infty$-regular} if for any infinite closed subset $F\subset X$ and point $x\in X\setminus F$ there exist disjoint open sets $V,U\subset X$ such that $x\in V$ and $U\cap F$ is infinite.

\begin{proposition} Each subspace $X$ of a countably compact Hausdorff space $Y$ is weakly $\infty$-regular.
\end{proposition}

\begin{proof} Given an infinite closed subset $F\subset X$ and a point $x\in X\setminus F$, consider the closure $\bar F$ of $F$ in $Y$ and observe that $x\notin\bar F$. By the countable compactness of $Y$, the infinite set $F$ has an accumulation point $y\in \bar F$. By the Hausdorff property of $Y$, there are two disjoint open sets $V,U\subset Y$ such that $x\in V$ and $y\in U$. Since $y$ is an accumulation point of the set $F$, the intersection $F\cap U$ is infinite. Then $V\cap X$ and $U\cap X$ are two disjoint open sets in $X$ such that $x\in V\cap X$ and $F\cap U\cap X$ is infinite, witnessing that the space $X$ is weakly $\infty$-regular.
\end{proof}

A subset $D$ of a topological space $X$ is called
\begin{itemize}
\item {\em discrete} if each point $x\in D$ has a neighborhood $O_x\subset X$ such that $D\cap O_x=\{x\}$;
\item {\em strictly discrete} if each point $x\in D$ has a neighborhood $O_x\subset X$ such that the family $(O_x)_{x\in D}$ is disjoint in the sense that $O_x\cap O_y=\emptyset$ for any distinct points $x,y\in D$;
\item {\em strongly discrete} if each point $x\in D$ has a neighborhood $O_x\subset X$ such that the family $(O_x)_{x\in D}$ is disjoint and locally finite in $X$.
\end{itemize}
It is clear that for every subset $D\subset X$ we have the implications
$$\mbox{strongly discrete $\Ra$ strictly discrete $\Ra$ discrete}.$$

\begin{theorem}\label{t:N2} Let $X$ be a subspace of a countably compact Hausdorff space $Y$. Then each infinite subset $I\subset X$ contains an infinite subset $D\subset I$ which is strictly discrete in $X$.
\end{theorem}

\begin{proof} By the countable compactness of $Y$, the set $I$ has an accumulation point $y\in Y$. Choose any point $x_0\in I\setminus\{y\}$ and using the Hausdorff property of $Y$, find a disjoint open neighborhoods $V_0$ and $U_0$ of the points $x_0$ and $y$, respectively. Choose any point $y_1\in U_0\cap I\setminus\{y\}$ and using the Hausdorff property of $Y$ choose open disjoint neighborhoods $V_1\subset U_0$ and $U_1\subset U_0$ of the points $x_1$ and $y$, respectively. Proceeding by induction, we can construct a sequence $(x_n)_{n\in\w}$ of points of $X$ and sequences $(V_n)_{n\in\w}$ and $(U_n)_{n\in\w}$ of open sets in $Y$ such that for every $n\in\IN$ the following conditions are satisfied:
\begin{itemize}
\item[1)] $x_n\in V_n\subset U_{n-1}$;
\item[2)] $y\in U_n\subset U_{n-1}$;
\item[3)] $V_n\cap U_n=\emptyset$.
\end{itemize}
The inductive conditions imply that the sets $V_n$, $n\in\w$, are pairwise disjoint, witnessing that the set $D=\{x_n\}_{n\in\w}\subset I$ is strictly discrete in $X$.
\end{proof}

For closed discrete subspaces in Lindel\"of subspaces, the strict discreteness of the set $D$ in Theorem~\ref{t:N2} can be improved to the strong discreteness. Let us recall that a topological space $X$ is {\em Lindel\"of} if each open cover of $X$ contains a countable subcover.

\begin{theorem}\label{t:N3} Let $X$ be a Lindel\"of subspace of a countably compact Hausdorff space $Y$. Then each infinite closed discrete subset $I\subset X$ contains an infinite subset $D\subset I$ which is strongly discrete in $X$.
\end{theorem}

\begin{proof} By the countable compactness of $Y$, the set $I$ has an accumulation point $y\in Y$. Since $I$ is closed and discrete in $X$, the point $y$ does not belong to the space $X$. By the Hausdorff property of $Y$, for every $x\in X$ there  are disjoint open sets $V_x,W_x\subset Y$ such that $x\in V_x$ and $y\in W_x$. Since the space $X$ is Lindel\"of, the open cover $\{V_x:x\in X\}$ has a countable subcover $\{V_{x_n}\}_{n\in\w}$. For every $n\in\w$ consider the open neighborhood $W_n=\bigcap_{k\le n}W_{x_k}$ of $y$.

 Choose any point $y_0\in I\setminus\{y\}$ and using the Hausdorff property of $Y$, find a disjoint open neighborhoods $V_0$ and $U_0\subset W_0$ of the points $y_0$ and $y$, respectively. Choose any point $y_1\in U_0\cap W_1\cap I\setminus\{y\}$ and using the Hausdorff property of $Y$ choose open disjoint neighborhoods $V_1\subset U_0$ and $U_1\subset U_0\cap W_1$ of the points $y_1$ and $y$, respectively. Proceeding by induction, we can construct a sequence $(y_n)_{n\in\w}$ of points of $X$ and sequences $(V_n)_{n\in\w}$ and $(U_n)_{n\in\w}$ of open sets in $Y$ such that for every $n\in\IN$ the following conditions are satisfied:
\begin{itemize}
\item[1)] $y_n\in V_n\subset U_{n-1}\cap W_n$;
\item[2)] $y\in U_n\subset U_{n-1}\cap W_n$;
\item[3)] $V_n\cap U_n=\emptyset$.
\end{itemize}
The inductive conditions imply that the family $(V_n)_{n\in\w}$ are pairwise disjoint, witnessing that the set $D=\{y_n\}_{n\in\w}\subset I$ is strictly discrete in $X$.
To show that $D$ is strongly discrete, it remains to show that the family $(V_n)_{n\in\w}$ is locally finite in $X$. Given any point $x\in X$, find $n\in\w$ such that $x\in V_{x_n}$ and observe that for every $i>n$ we have $V_i\cap V_{x_n}\subset W_i\cap V_{x_n}\subset W_{n}\cap V_{x_n}=\emptyset$.
\end{proof}

A topological space $X$ is called {\em $\ddot\w$-regular} if it for any closed discrete subset $F\subset X$ and point $x\in X\setminus F$ there exist disjoint open sets $U_F$ and $U_x$ in $X$ such that $F\subset U_F$ and $x\in U_x$.


\begin{proposition}\label{p:sd} Each countable closed discrete subset $D$ of a (Lindel\"of) $\ddot\w$-regular $T_1$-space $X$ is strictly discrete (and strongly discrete) in $X$.
\end{proposition}

\begin{proof} The space $X$ is Hausdorff, being an $\ddot\w$-regular $T_1$-space. If the subset $D\subset X$ is finite, then $D$ is strongly discrete by the Hausdorff property of $X$. So, assume that $D$ is infinite and hence $D=\{z_n\}_{n\in\w}$ for some pairwise distinct points $z_n$. By the $\ddot\w$-regularity there are two disjoint open sets $V_0,W_0\subset X$ such that $z_0\in V_0$ and $\{z_n\}_{n\ge 1}\subset W_0$.

Proceeding by induction, we can construct sequences of open sets $(V_n)_{n\in\w}$ and $(W_n)_{n\in\w}$ in $X$ such that for every $n\in\w$ the following conditions are satisfied:
\begin{itemize}
\item $z_n\in V_n\subset W_{n-1}$;
\item $\{z_k\}_{k>n}\subset W_n\subset W_{n-1}$;
\item $V_n\cap W_n=\emptyset$.
\end{itemize}
These conditions imply that the family $(V_n)_{n\in\w}$ is disjoint, witnessing that the set $D$ is strictly discrete in $X$.

Now assume that the space $X$ is Lindel\"of and let $V=\bigcup_{n\in\w}V_n$. By the $\ddot\w$-regularity of $X$, each point $x\in X\setminus V$ has a neighborhood $O_x\subset X$ whose closure $\bar O_x$ does not intersect the closed discrete subset $D$ of $X$. Since $X$ is Lindel\"of, there exists a countable set $\{x_n\}_{n\in\w}\subset X\setminus V$ such that $X=V\cup \bigcup_{n\in\w}O_{x_n}$. For every $n\in \w$ consider the open neighborhood $U_n:=V_n\setminus\bigcup_{k\le n}\bar O_{x_k}$ of $z_n$ and observe that the family $(U_n)_{n\in\w}$ is disjoint and locally finite in $X$, witnessing that the set $D$ is strongly discrete in $X$.
\end{proof}

The following proposition shows that the property described in Theorem~\ref{t:N2} holds for $\ddot\w$-regular  spaces.

\begin{proposition} Every infinite subset $I$ of an $\ddot\w$-regular $T_1$-space $X$ contains an infinite subset $D\subset I$, which is strictly discrete in $X$.
\end{proposition}

\begin{proof} If $I$ has an accumulation point in $X$, then a strictly discrete infinite subset can be constructed repeating the argument of the proof of Theorem~\ref{t:N2}. So, we assume that $I$ has no accumulation point in $X$ and hence $I$ is closed and discrete in $X$. Replacing $I$ by a countable infinite subset of $I$, we can assume that $I$ is countable. By Proposition~\ref{p:sd}, the set $I$ is strictly discrete in $X$.
\end{proof}

A topological space $X$ is called {\em superconnected} \cite{BMT} if for any non-empty open sets $U_1,\dots, U_n$ the intersection $\overline{U}_1\cap\dots\cap\overline{U}_n$ is not empty. It is clear that a superconnected space containing more than one point is not regular. An example of a superconnected second-countable Hausdorff space can be found in \cite{BMT}.

\begin{proposition} Any first-countable superconnected Hausdorff  space $X$ with $|X|>1$ contains an infinite set $I\subset X$ such that each infinite subset $D\subset I$ is not  strictly discrete in $X$.
\end{proposition}

\begin{proof} For every point $x\in X$ fix a countable neighborhood base $\{U_{x,n}\}_{n\in\w}$ at $x$ such that $U_{x,n+1}\subset U_{x,n}$ for every $n\in\w$.

Choose any two distinct points $x_0,x_1\in X$ and for every $n\ge 2$ choose a point $x_n\in\bigcap_{k<n}\overline{U}_{x_k,n}$. We claim that the set $I=\{x_n\}_{n\in\w}$ is infinite. In the opposite case, we use the Hausdorff property and find a neighborhood $V$ of $x_0$ such that $\overline{V}\cap I=\{x_0\}$. Find $m\in\w$ such that $U_{x_0,m}\subset V$ and $x_0\notin \overline{U}_{x_1,m}$. Observe that $$x_m\in I\cap \overline{U}_{x_0,m}\cap\overline{U}_{x_1,m}=\{x_0\}\cap \overline{U}_{x_1,m}=\emptyset,$$ which is
a desired contradiction showing that the set $I$ is infinite.

Next, we show that any infinite subset $D\subset I$ is not strictly discrete in $X$. To derive a contradiction, assume that $D$ is strictly discrete. Then each point $x\in D$ has a neighborhood $O_x\subset X$ such that the family $(O_x)_{x\in D}$ is disjoint. Choose any point $x_k\in D$ and find $m\in\w$ such that $U_{x_k,m}\subset O_{x_k}$. Replacing $m$ by a larger number, we can assume that $m>k$ and $x_m\in D$. Since $x_m\in\overline{U}_{x_k,m}\subset \overline O_{x_k}$, the intersection $O_{x_m}\cap O_{x_k}$ is not empty, which contradicts the choice of the neighborhoods $O_x$, $x\in D$.
\end{proof}

Next, we establish one property of subspaces of functionally Hausdorff countably compact spaces. We recall that a topological space $X$ is {\em functionally Hausdorff} if for any distinct points $x,y\in X$ there exists a continuous function $f:X\to [0,1]$ such that $f(x)=0$ and $f(x)=1$.

A subset $U$ of a topological space $X$ is called {\em functionally open} if $U=f^{-1}(V)$ for some continuous function $f:X\to\IR$ and some open set $V\subset\IR$.

A subset $K\subset X$ of topological space is called {\em functionally compact} if each open cover of $K$ by functionally open subsets of $X$ has a finite subcover.

\begin{proposition} If $X$ is a subspace of a functionally Hausdorff countably compact space $Y$, then no infinite closed discrete subspace $D\subset X$ is contained in a functionally compact subset of $X$.
\end{proposition}

\begin{proof} To derive a contradiction, assume that $D$ is contained in a functionally compact subset $K$ of $X$. By the countable compactness of $Y$, the set $D$ has an assumulation point $y\in Y$. Since $D$ is closed and discrete in $X$, the point $y$ does not belong to $X$ and hence $y\notin K$. Since $Y$ is functionally Hausdorff, for every $x\in K$ there exists a continuous function $f_x:Y\to[0,1]$ such that $f_x(x)=0$ and $f_x(y)=1$. By the functional compactness of $K$, the cover $\{f_x^{-1}([0,\frac12)):x\in K\}$ contains a finite subcover $\{f_x^{-1}([0,\frac12)):x\in E\}$ where $E$ is a finite subset of $K$. Then $D\subset K\subset f^{-1}([0,\frac12))$ for the continuous function $f=\max_{x\in E}f_x:Y\to [0,1]$, and $f^{-1}((\frac12,1])$ is a neighborhood of $y$, which is disjoint with the set $D$. But this is not possible as $y$ is an accumulation point of $D$.
\end{proof}

Finally, we construct an example of a regular separable first-countable scattered space that embeds into a Hausdorff countably compact space but does not embed into Urysohn countably compact spaces. We recall that a topological space $X$ is {\em Urysohn} if any distinct points of $X$ have disjoint closed neighborhoods in $X$.

\begin{example}\label{e:d} There exists a topological space $X$ such that
\begin{enumerate}
\item $X$ is regular, separable, and first-countable;
\item $X$ can be embedded into a Hausdorff totally countably compact space;
\item $X$ cannot be embedded into an Urysohn countably compact space.
\end{enumerate}
\end{example}

\begin{proof} In the construction of the space $X$ we shall use almost disjoint dominating subsets of $\w^\w$. Let us recall \cite{vD} that a subset $D\subset\w^\w$ is called {\em dominating} if for any  $x\in\w^\w$ there exists $y\in D$ such that $x\le^* y$, which means that $x(n)\le y(n)$ for all but finitely many numbers $n\in\w$.
By $\mathfrak d$ we denote the smallest cardinality of a dominating subset $D\subset\w^\w$. It is clear that $\w_1\le\mathfrak d\le\mathfrak c$.

We say that a family of function $D\subset\w^\w$ is {\em almost disjoint} if for any distinct $x,y\in D$ the intersection $x\cap y$ is finite. Here we identify a function $x\in \w^\w$ with its graph $\{(n,x(n)):n\in\w\}$ and hence identify the set of functions $\w^\w$ with a subset of the family $[\w\times\w]^\w$ of all infinite subsets of $\w\times\w$ .

\begin{claim}\label{cl1} There exists an almost disjoint dominating subset $D\subset\w^\w$ of cardinality $|D|=\mathfrak d$.
\end{claim}

\begin{proof} By the definition of $\mathfrak d$, there exists a dominating family $\{x_\alpha\}_{\alpha\in\mathfrak d}\subset \w^\w$. It is well-known that $[\w]^\w$ contains an almost disjoint family $\{A_\alpha\}_{\alpha\in\mathfrak c}$ of cardinality continuum. For every $\alpha<\mathfrak d$ choose a strictly increasing  function $y_\alpha:\w\to A_\alpha$ such that $x_\alpha\le y_\alpha$. Then the set $D=\{y_\alpha\}_{\alpha\in \mathfrak d}$ is dominating and almost disjoint.
\end{proof}

By Claim~\ref{cl1}, there exists an almost disjoint  dominating subset $D\subset\w^\w\subset[\w\times\w]^\w$. For every $n\in\w$ consider the vertical line $\lambda_n=\{n\}\times\w$ and observe that the family $L=\{\lambda_n\}_{n\in\w}$ is disjoint and the family  $D\cup L\subset[\w\times\w]^\w$ is almost disjoint.

Consider the space $Y=(D\cup L)\cup(\w\times\w)$ endowed with the topology consisting of the sets $U\subset Y$ such that for every $y\in (D\cup L)\cap U$ the set $y\setminus U\subset\w\times\w$ is finite. Observe that all points in the set $\w\times\w$ are isolated in $Y$. Using the almost disjointness of the family $D\cup L$, it can be shown that the space $Y$ is regular, separable, locally countable, scattered and locally compact.

Choose any point $\infty\notin \w\times Y$ and consider the space $Z=\{\infty\}\cup(\w\times Y)$ endowed with the topology consisting of the sets $W\subset Z$ such that
\begin{itemize}
\item for every $n\in\w$ the set $\{y\in Y:(n,y)\in W\}$ is open in $Y$, and
\item if $\infty\in W$, then there exists $n\in\w$ such that $\bigcup_{m\ge n}\{m\}\times Y\subset W$.
\end{itemize}
It is easy to see $Z=\{\infty\}\cup(\w\times Y)$ is first-countable, separable, scattered and regular.

Let $\sim$ be the smallest equivalence relation on $Z$ such that $$\mbox{$(2n,\lambda)\sim(2n+1,\lambda)$ and $(2n+1,d)\sim (2n+2,d)$}
$$for any $n\in\w$, $\lambda\in L$ and $d\in D$.

Let $X$ be the quotient space $Z/_\sim$ of $Z$ by the equivalence relation $\sim$. It is easy to see that the equivalence relation $\sim$ has at most two-element equivalence classes and the quotient map $q:Z\to X$ is closed and hence perfect. Applying \cite[3.7.20]{Eng}, we conclude that the space $X$ is regular. It is easy to see that $X$ is separable, scattered and first-countable.
It remains to show that $X$ has the properties (2), (3) of Example~\ref{e:d}.
This is proved in the following two claims.

\begin{claim} The space $X$ does not admit an embedding into an Urysohn countably compact space.
\end{claim}

\begin{proof} To derive a contradiction, assume that $X=q(Z)$ is a subspace of an Urysohn countably compact space $C$. By the countable compactness of $C$, the set $q(\{0\}\times L)\subset X\subset C$ has an accumulation point $c_0\in C$. The point $c_0$ is distinct from $q(\infty)$, as $q(\infty)$ is not an accumulation point of the set $q(\{0\}\times L)$ in $X$. Let $l\in\w$ be the largest number such that $c_0$ is an accumulation point of the set $q(\{l\}\times L)$ in $C$.

Let us show that the number $l$ is well-defined. Indeed, by the Hausdorffness of the space $C$, there exists a neighborhood $W\subset C$ of $q(\infty)$ such that $c_0\not\subset\overline{W}$. By the definition of the topology of the space $Z$, there exists $m\in\w$ such that $\bigcup_{k\ge m}\{k\}\times Y\subset q^{-1}(W)$. Then $c_0$ is not an accumulation point of the set $\bigcup_{k\ge m}q(\{k\}\times L)$ and hence the number $l$ is well-defined and $l<m$.

The definition of the equivalence relation $\sim$ implies that the number $l$ is odd. By the countable compactness of $C$, the infinite set $q(\{l+1\}\times L)$ has an accumulation point $c_1\in C$. The maximality of $l$ ensures that $c_1\ne c_0$. By the Urysohn property of $C$, the points $c_0,c_1$ have open neighborhoods $U_0,U_1\subset C$ with disjoint closures in $C$.

For every $i\in\{0,1\}$ consider the set $J_i=\{n\in\w:q(l+i,\lambda_n)\in U_i\}$, which is infinite, because $c_i$ is an accumulation point of the set $q(\{l+i\}\times L)=\{q(l+i,\lambda_n):n\in\w\}$. For every $n\in J_i$ the open set $q^{-1}(U_i)\subset Z$ contains the pair $(l+i,\lambda_n)$. By the definition of the topology at $(l+i,\lambda_n)$, the set $(\{l+i\}\times \lambda_n)\setminus q^{-1}(U_i)\subset \{l+i\}\times\{n\}\times\w$ is finite and hence is contained in the set $\{l+i\}\times\{n\}\times[0,f_i(n)]$ for some number $f_i(n)\in\w$. Using the dominating property of the family $D$, choose a function $f\in D$ such that $f(n)\ge f_i(n)$ for any $i\in\{0,1\}$ and $n\in J_i$. It follows that for every $i\in\{1,2\}$ the set $\{l+i\}\times f\subset\{l+i\}\times(\w\times\w)$ has infinite intersections with the preimage $q^{-1}(U_i)$ and hence $\{(l+i,f)\}\in\overline{q^{-1}(U_i)}\subset q^{-1}(\overline{U}_i)$. Taking into account that the number $l$ is odd,
we conclude that $$q(l,f)=q(l+1,f)\in\overline{U}_0\cap\overline{U}_1=\emptyset.$$
which is a desired contradiction completing the proof of the claim.
\end{proof}

\begin{claim} The space $X$ admits an embedding into a Hausdorff totally countably compact space.
\end{claim}

\begin{proof} Using the Kuratowski-Zorn Lemma, enlarge the almost disjoint family $D\cup L$ to a maximal almost disjoint family $M\subset[\w\times\w]^\w$.
Consider the space $Y_M=M\cup(\w\times\w)$ endowed with the topology consisting of the sets $U\subset Y_M$ such that for every $y\in M\cap U$ the set $y\setminus U\subset\w\times\w$ is finite. It follows that $Y_M$ is a regular locally compact first-countable space, containing $Y$ as an open dense subspace.
The maximality of $M$ implies that each sequence in $\w\times\w$ contains a subsequence that converges to some point of the space $Y_M$. This property implies that the subspace $\tilde Y:=(W_\omega M)\cup(\w\times\w)$ of the Wallman extension of $W(Y_M)$ is totally countably compact. Repeating the argument from Example~\ref{e3}, one can show that the space $\tilde Y$ is Hausdorff.

Let $\tilde Z=\{\infty\}\cup(\w\times\tilde Y)$ where $\infty\notin\w\times\tilde Y$. The space $\tilde Z$ is endowed with the topology consisting of the sets
$W\subset \tilde Z$ such that
\begin{itemize}
\item for every $n\in\w$ the set $\{y\in \tilde Y:(n,y)\in W\}$ is open in $\tilde Y$, and
\item if $\infty\in W$, then there exists $n\in\w$ such that $\bigcup_{m\ge n}\{m\}\times \tilde Y\subset W$.
\end{itemize}
Taking into account that the space $\tilde Y$ is Hausdorff and totally countably compact, we can prove that so is the the space $\tilde Z$.

Let $\sim$ be the smallest equivalence relation on $\tilde Z$ such that $$\mbox{$(2n,\lambda)\sim(2n+1,\lambda)$ and $(2n+1,d)\sim (2n+2,d)$}
$$for any $n\in\w$, $\lambda\in W_\omega L$ and $d\in W_\omega D$.

Let $\tilde X$ be the quotient space $\tilde Z/_\sim$ of $\tilde Z$ by the equivalence relation $\sim$. It is easy to see that the space $\tilde X$ is Hausdorff, totally countably compact and contains the space $X$ as a dense subspace.
\end{proof}
\end{proof}

However, we do not know the answer on the following intriguing problem:

\begin{problem}
Is it true that each (scattered) regular topological space can be embedded into a Hausdorff countably compact topological space?
\end{problem}


\end{document}